\documentclass[11pt]{amsart}

\usepackage[T1]{fontenc}
\usepackage[utf8]{inputenc}
\usepackage[english]{babel}

\usepackage{pifont,geometry}
\usepackage{amsmath,amssymb,amsthm}

\usepackage{algorithm,algpseudocode}
\usepackage{graphicx}
\usepackage{enumitem}

\usepackage[x11names,rgb,table]{xcolor}
\usepackage{tikz}
\usetikzlibrary{arrows,shapes}

\usepackage[pdftex,bookmarks,colorlinks]{hyperref}

\newtheorem{theorem}{Theorem}[section]
\newtheorem{corollary}[theorem]{Corollary}
\newtheorem{proposition}[theorem]{Proposition}
\newtheorem{lemma}[theorem]{Lemma}

\theoremstyle{definition}
\newtheorem{definition}[theorem]{Definition}

\numberwithin{equation}{section}

\newcommand{\NN}{\mathbb N}
\newcommand{\ZZ}{\mathbb Z}

\newcommand{\round}[1]{\ensuremath{\lfloor#1\rceil}}

\begin{document} 	
 	\title{On the decoding of 1-Fibonacci error correcting codes}
 	
 	\author{Emanuele Bellini}
 	\address{\textnormal{Emanuele Bellini}. Technology Innovation Institute, Abu Dhabi, UAE}
 	\email{{eemanuele.bellini@gmail.com}}
 	
 	\author{Chiara Marcolla}
 	\address{\textnormal{Chiara Marcolla}. Technology Innovation Institute, Abu Dhabi, UAE}
 	\email{{chiara.marcolla@gmail.com}}
 	
 	\author{Nadir Murru}
 	\address{\textnormal{Nadir Murru}. University of Torino, Department of Mathematics, Torino, Italy}
 	\email{{nadir.murru@unito.it}}

 	\keywords{Fibonacci numbers, Diophantine equations, error correcting codes, Golden Ratio, recurrent sequence}
 	
 	
\begin{abstract}
		The study of new error correcting codes has raised attention in the last years, 
		especially because of their use in cryptosystems 
		that are resistant to attacks running on quantum computers.
		In 2006, while leaving a more in-depth analysis for future research, 
		Stakhov gave some interesting ideas on how to exploit Fibonacci numbers to derive an original error correcting code with a compact representation.
		In this work 
		we provide an explicit formula to compute the redundancy of Stakhov codes, 
		we identify some flows in the initial decoding procedure described by Stakhov, 
		whose crucial point is to solve some non-trivial Diophantine equations,  
		and provide a detailed discussion on how to avoid solving such equations
		in some cases and on how to detect and correct errors more efficiently.
\end{abstract}

\maketitle

\section{Introduction} \label{sec:intro}

Coding theory has been playing an important role in the field of communication for many years.
Many kind of codes for detecting and correcting possible errors occurred in the transmission of a message have been developed. 
Moreover, the application of error correcting codes to the development of cryptographic systems resistant to quantum attacks has further increased their interest and study.
The main drawback of using error correcting codes in cryptography, 
is that they yield very large keys, signatures or encryptions.
For these reasons the study of new codes is a very active research field, 
and for the case of cryptography, 
the problem of finding a code with a compact representation is still open. 

In 2005, 
Stakhov \cite{Stakhov} introduced the so called $p$-Fibonacci error correcting codes, 
whose decoding exploits the properties of the Fibonacci numbers, their generalizations, golden ratio approximation and Diophantine equations. 
Stakhov claims that the new error correction method has a correction ability of $93.33\%$
and  that such codes have a redundancy of $33.3\%$. 
Such claims raise some suspicions
from an information theoretic point of view. 
Thus, 
a deeper and more precise analysis of Stakhov's results is necessary, both concerning the redundancy and the ability of detecting and correcting errors. 
Indeed, in Section~\ref{sec:errors}, 
we are going to explain that 
some methodologies and conclusions 
presented by Stakhov are misleading and incorrect.
From a practical point of view, it is hard to
think of an actual channel in which 
Stakhov codes would be convenient. 
Yet, the use of these codes might result useful in the context of code-based cryptography, 
where there are no restriction on the channel, and
where a compact representation of the public key, 
which is often given by the generator matrix, 
is paramount.

In this paper, we partially address the issues mentioned above 
in the case of 1-Fibonacci codes.
We first show how to improve the error detection phase.
We then show how to correct some type of errors 
without the need of solving potentially complex Diophantine equations.
We also show how to restrict the message space in order not to have
ambiguities in the decoding.
And finally, we provide an explicit formula for the redundancy of these codes.

The paper is structured as follows.
In Section \ref{sec:related}, we give a brief overview about works that looked into the use of $p$-Fibonacci numbers in coding theory. 
In Section~\ref{sec:fibo-code}, 
we fix also the notation and 
recall the idea of the $p$-Fibonacci encoding method, 
while the decoding is presented in Section \ref{sec:fibo-decoding}.
Section \ref{sec:errors} and Section \ref{sec:redundancy} are devoted to our contributions.
To conclude, in Section \ref{sec:conc} we summarize our results and point to future research goals.

\section{Related works} \label{sec:related}

The Fibonacci numbers $(F_n)_{n=0}^{+\infty}$ are one of the most famous linear recurrent sequences, defined as
\[\begin{cases} F_1 = 1, \quad F_2 = 1 \cr F_{n+1} = F_n + F_{n-1}, \quad \forall n > 2   \end{cases}.\]
This linear recurrent sequence of order 2 can be generalized to higher orders as follows:
\begin{equation} \label{eq:pfibo}\begin{cases} F_1^{(p)} = \ldots = F_{p+1}^{(p)} = 1 \cr F_n^{(p)} = F_{n-1}^{(p)} + F_{n-p-1}^{(p)}, \quad \forall n>p+1 \end{cases},\end{equation}
which is a linear recurrent sequence (whose elements can be called $p$-Fibonacci numbers) of order $p+1$ with characteristic polynomial $x^{p+1} - x^p -1$. 

The $p$-Fibonacci coding/decoding method introduced by Stakhov \cite{Stakhov} has been studied in further works. In \cite{StakhovBook}, Stakhov gives some more details about the $p$-Fibonacci code, providing in particular an estimation of the redundancy. However, in this evaluation he does not take into account the increase of the dimensions of the messages due to the matrix multiplication involved in the encoding method, but he only considers the transmission of the determinant of a matrix used for representing the message that must be sent. 

While Stakhov mainly focused on the case $p=1$ corresponding to the classical Fibonacci numbers, in \cite{Esma} the authors proposed an analysis about the computational complexity of the coding/decoding method in the general case, proposing also some numerical examples. The authors show that in the worst case error correction can be done in time $O(2^{p^2})$, but they specify that the worst case is not a serious problem since it does not occur in practice, considering the channel characteristics.
The authors showed also that the code has a rate of error correction of $\frac{2^{p^2}-1}{2^{p^2}}$. However, in the whole discussion, the authors do not address the problem of the solution of some Diophantine equations involved in the step of the errors correction, which can be very time consuming if some adjustments are not taken, as we will see in the next sections. 

In \cite{Basu}, the authors give another (similar) estimation for the rate of error correction, that is $\frac{2^{p^2}-2}{2^{p^2}-1}$, and they propose a specific discussion on the case $p=2$. Also in this paper, the authors take into considerations only the matricial operations involved in the error correction phase, without discussing the use of the Diophantine equations. 

The $p$-Fibonacci code can be also generalized considering some linear recurrent sequences, like the $p$-Lucas numbers \cite{Pra3} or other ones that generalize the $p$-Fibonacci numbers, as in \cite{Basu2} and \cite{Pra2}. Prasad \cite{Pra} proposed also the use of recurrent sequences of polynomials $(F_n^{(p)}(x,y))$.
Further generalizations can be found in  \cite{Basu3}, \cite{Basu4}, \cite{Esma2}, \cite{Tas}.

\section{Preliminaries on the Fibonacci encoding method} \label{sec:fibo-code}

In this section we present the coding method proposed by Stakhov \cite{Stakhov}. For the seek of simplicity, in the following we focus on the case $p=1$, i.e., we focus on the classical Fibonacci numbers. 
It is fairly easy to generalize the results for larger $p$.
The Fibonacci sequence $(F_n)$ has many beautiful and interesting properties. Here, we only recall the properties useful for constructing the coding method:

\begin{itemize}
	
	\item Given $Q = \begin{pmatrix}  1 & 1 \cr 1 & 0 \end{pmatrix}$, we have $Q^n = \begin{pmatrix} F_{n+1} & F_n \cr F_n & F_{n-1} \end{pmatrix}$, for all $n > 1$.
	
	\item $\det Q^n = (-1)^n$ that is the Cassini identiy $F_{n+1}F_{n-1} - F_n^2 = (-1)^n$ for all $n > 1$.
	
	\item $Q^{-2k} = \begin{pmatrix} F_{2k-1} & -F_{2k} \cr -F_{2k} & F_{2k+1} \end{pmatrix}$ for all $k > 1$, $Q^{-(2k+1)} = \begin{pmatrix} -F_{2k} & F_{2k+1} \cr F_{2k+1} & -F_{2k+2} \end{pmatrix}$ for all $k > 0$.
	
	\item $\lim_{n \rightarrow +\infty} \cfrac{F_{n+1}}{F_n} = \varphi$, where $\varphi$ is the golden mean, i.e., the root greatest in modulo of $x^2 - x - 1$.
	
\end{itemize}
Given a message $m$ as an integer number, we represent it by means of a $2 \times 2$ matrix
$$M = \begin{pmatrix} m_1 & m_2 \cr m_3 & m_4  \end{pmatrix}$$
whose elements are positive non--zero integers, i.e., $m$ is identified by the matrix $M$. The encoding procedure consists in the multiplication of the matrix $M$ by a coding matrix that is the matrix $Q^n$ for a fixed integer $n$:
\begin{equation} \label{eq:coding} C = M \cdot Q^n = \begin{pmatrix} c_1 & c_2 \cr c_3 & c_4  \end{pmatrix} \end{equation}
and consequently
{\small{
$$
c_1 = F_{n+1} m_1 + F_n m_2, \,
c_2 = F_n m_1 + F_{n-1} m_2, \,
c_3 = F_{n+1} m_3 + F_n m_4, \, 
c_4 = F_n m_3 + F_{n-1} m_4  
.$$}}
In the following, when we refer to messages and codewords, we refer to $2 \times 2$ matrices with integral entries.
Hence, in order to detect and correct possible errors, the codeword $C$ is sent together with $\det M$, 
which plays an important role in the error correction, as we will see in the next sections.

\section{The Fibonacci Decoding method}
\label{sec:fibo-decoding}

Given a codeword $C$, a receiver can decode $C$ to retrieve the original message $M$ by means of the following multiplication:
\begin{equation} \label{eq:decode} C \cdot Q^{-n} = M. \end{equation}
Now, we recall the consideration of Stakhov \cite{Stakhov} for detecting and correcting possible errors. 

\begin{proposition} \label{prop:no-err}
	Given a message $M$, if the codeword $C = M \cdot Q^n$ does not contain errors, then 
	\begin{enumerate}
		\item $\det C = (-1)^n \det M$;
		\item $\cfrac{c_1}{c_2} \approx \varphi$, $\cfrac{c_3}{c_4} \approx \varphi$;
		\item $c_1 = \cfrac{(-1)^n \det M + c_2 c_3}{c_4}$, $c_2 = \cfrac{(-1)^n \det M + c_1 c_4}{c_3}$,
		$c_3 = \cfrac{(-1)^n \det M + c_1 c_4}{c_2}$, \\ $c_4 = \cfrac{(-1)^n \det M + c_2 c_3}{c_1}$ are integer numbers.
	\end{enumerate}
	See \cite{Stakhov}.
\end{proposition}

The quality of the approximations, expressed in the point 2 of the above proposition, plays an important role in the error correction and can be explicitly determined, as we will see in the next section.\\
In the procedure proposed by Stakhov \cite{Stakhov}, a checking element, namely $\det M$, is sent to the receiver together with the codeword $C$. As in  \cite{Stakhov}, 
we suppose that the checking element arrives to the receiver without errors. Hence,
if $\det C \not= (-1)^n \det M$, then the codeword contains one or more errors, otherwise we conclude that the codeword does not contain errors and the message $M$ 
can be recovered by \eqref{eq:decode}. This means that this method is not able to recognize errors of the type $U \cdot C$ or $C \cdot U$ where $\det U = 1$.
In the case that $\det C \not= (-1)^n \det M$, the following quantities must be evaluated:
\begin{align} 
	\label{eq:xi} 
	x_1 & = \cfrac{(-1)^n \det M + c_2 c_3}{c_4},
	\quad
	x_2 = \cfrac{(-1)^n \det M + c_1 c_4}{c_3}, 
	\\
	x_3 & = \cfrac{(-1)^n \det M + c_1 c_4}{c_2},
	\quad
	x_4 = \cfrac{(-1)^n \det M + c_2 c_3}{c_1}
	\,.
\end{align}
If one and only one of the quantities $x_i$ is an integer, then the codeword $C$ contains one error, which is the element $c_i$ and it is corrected by $x_i$. If the codeword contains more than one error, then all the quantities $x_i$ are not integers. In this case, Stakhov \cite{Stakhov} recommends to solve some Diophantine equations and use point 2 in Proposition~\ref{prop:no-err} in order to correct the errors. 
In particular, he suggests to check first all the hypotheses of double errors yielding to solve a Diophantine equation of the kind
$$x c_4 - y c_3 = (-1)^n \det M$$
if, for instance, we are checking the hypothesis that $c_1$ and $c_2$ are wrong. Among the infinite solutions of the above Diophantine equation, the correct values of $c_1$ and $c_2$ are provided by the solutions that satisfy the approximation property in Proposition~\ref{prop:no-err}. Stakhov \cite{Stakhov} stated that with a similar procedure all the cases of two and three errors in $C$ can be corrected. The method does not work only when all the $c_i$'s are wrong. 

We would like to point out that in \cite{Stakhov}, as the same author states, an in--depth analysis about detection and correction of errors is not performed, but the author only aimed at giving general considerations and ideas. In the other works dealing with this coding theory (such as \cite{Basu}, \cite{Basu2}, \cite{Esma}, \cite{Pra}, \cite{Pra2}), an in--depth discussion is still missing. Specifically, we would like to highlight that solving a Diophantine equation is not a simple task and it can be very time consuming. Moreover, without some specific conditions on the messages and some detailed considerations on the quality of the approximation in Proposition \ref{prop:no-err}, it could be impossible to detect the correct message among the infinite solutions of a Diophantine equation. 
We aim at performing a detailed discussion about this issues in the next sections, providing some considerations that improve the detection and correction of errors in the Fibonacci code and make also possible to correct errors which otherwise would not be correctable.
Finally, as one can see, 
the correctness of $\det M$ 
plays a fundamental role in the decoding technique.
From a practical point of view,
it is hard to think of a real channel in which one can guarantee the correct transmission of $\det M$ only.
One solution proposed by Stakhov is to send $\det M$ using a standard error correction code. 
Though, this does not seem a very practical solution.


\section{A discussion about error detection and correction}
\label{sec:errors}

In this section we first prove properties on the quality of the approximation of the Golden ratio given by the received codeword.
Secondly, we show how to improve the error detection phase, by providing a condition that allows us to avoid the redundant checking of all the hypothesis of double error.
As a third result we show how to correct double errors not in the same row with simple calculations
instead of solving Diophantine equations.
And finally, we show how to restrict the message space in order not to have
ambiguities in the decoding of double errors on the same row.

We would like to warn the reader, that, when we speak about errors, we are not referring to \emph{bit} errors, 
as it is usually done in coding theory. 
We instead use Stakhov terminology, and 
refer to \emph{integer} errors, which usually include several bits.

\subsection{Preliminaries}

In the following, we will fix $n$ as an odd positive integer for the coding matrix $Q^n$ (similar results hold when $n$ even). In this case, given a message $M$ and the corresponding codeword $C = M \cdot Q^n$, we have
$$
m_1 = -F_{n-1} c_1 + F_n c_2, \quad 
m_2 = F_n c_1 - F_{n+1} c_2, \quad
m_3 = - F_{n-1} c_3 + F_n c_4, \quad 
m_4 = F_n c_3 - F_{n+1} c_4  
$$
from which we obtain the following inequalities:
\begin{equation} \label{eq:approx} \cfrac{F_{n+1}}{F_n} < \cfrac{c_1}{c_2} < \cfrac{F_n}{F_{n-1}}, \quad \cfrac{F_{n+1}}{F_n} < \cfrac{c_3}{c_4} < \cfrac{F_n}{F_{n-1}}   \end{equation}
which explicit the quality of the approximations stated in Proposition \ref{prop:no-err}. In the following, we prove some properties that can be exploited for providing a more precise strategy in the detection and correction of errors.

\begin{proposition} \label{prop:one-error}
	Let $h$ be an integer, if $m_1, m_2, m_3, m_4< F_{n-1}$, then
	\begin{align*}
		\cfrac{c_1 + h}{c_2} 
		&
		\in 
		\left( \cfrac{F_{n+1}}{F_n}, \cfrac{F_n}{F_{n-1}} \right), 
		\quad 
		\cfrac{c_1}{c_2+h} 
		\in 
		\left( \cfrac{F_{n+1}}{F_n}, \cfrac{F_n}{F_{n-1}} \right), 
		\\
		\cfrac{c_3 + h}{c_4} 
		&
		\in 
		\left(\cfrac{F_{n+1}}{F_n}, \cfrac{F_n}{F_{n-1}}\right), 
		\quad 
		\cfrac{c_3}{c_4+h} 
		\in
		\left(\cfrac{F_{n+1}}{F_n}, \cfrac{F_n}{F_{n-1}}\right)
	\end{align*}
	if and only if $h = 0$.
\end{proposition}
\begin{proof}
	If $h = 0$, by \eqref{eq:approx} we have 
	$\cfrac{c_1}{c_2} \in \left( \cfrac{F_{n+1}}{F_n}, \cfrac{F_n}{F_{n-1}} \right)$. Now, let us suppose
	$\cfrac{c_1 + h}{c_2} \in \left( \cfrac{F_{n+1}}{F_n}, \cfrac{F_n}{F_{n-1}} \right).$
	This means that 
	$\cfrac{m_1 F_{n+1} + m_2 F_n + h}{m_1 F_n + m_2 F_{n-1}} < \cfrac{F_n}{F_{n-1}}$
	and with simple calculations we obtain
	$m_1 (F_{n+1} F_{n-1} - F_n^2) + h F_{n-1} < 0,$
	i.e.
	$h < \cfrac{m_1 (-1)^{n-1}}{F_{n-1}}.$
	Similarly, from
	$\cfrac{m_1 F_{n+1} + m_2 F_n + h}{m_1 F_n + m_2 F_{n-1}} > \cfrac{F_{n+1}}{F_n}$
	we get
	$h > \cfrac{(-1)^n m_2}{F_n}.$
	Thus, if $m_1, m_2 < F_{n-1}$, the only integer in the interval $\left( \cfrac{(-1)^n m_2}{F_n}, \cfrac{(-1)^{n-1} m_1}{F_{n-1}}  \right)$ is 0.
	\end{proof}

\begin{proposition} \label{prop:two-errors}
	Given an integer $h$, we have that  
	$$\cfrac{c_1 + h}{c_2 + k} \in \left(\cfrac{F_{n+1}}{F_n}, \cfrac{F_n}{F_{n-1}}\right), \quad \cfrac{c_3 + h}{c_4 + k} \in \left(\cfrac{F_{n+1}}{F_n}, \cfrac{F_n}{F_{n-1}}\right)$$
	if and only if 
	$$\cfrac{h F_{n-1} - m_1}{F_n} < k < \cfrac{h F_n + m_2}{F_{n+1}}.$$
\end{proposition}
\begin{proof}
	Considering $\cfrac{c_1 + h}{c_2 + k} < \cfrac{F_n}{F_{n-1}}$ we get
	$\cfrac{m_1 F_{n+1} + m_2 F_n + h}{m_1 F_n + m_2 F_{n-1} + k} < \cfrac{F_n}{F_{n-1}}.$
	Supposing $k > - (m_1 F_n + m_2 F_{n-1})$, it is easy to obtain
	$h > \cfrac{h F_{n-1} - m_1}{F_n}.$
	Similarly, considering $\cfrac{c_1 + h}{c_2 + k} > \cfrac{F_{n+1}}{F_n}$ we obtain
	$k < \cfrac{h F_n + m_2}{F_{n+1}}.$
	Finally, note that the case $k < - (m_1 F_n + m_2 F_{n-1})$ does not yield any possible value for $k$ that satisfies $\cfrac{c_1 + h}{c_2 + k} \in \left(\cfrac{F_{n+1}}{F_n}, \cfrac{F_n}{F_{n-1}}\right)$.
	\end{proof}

\subsection{Improving the error detection phase}

In Section \ref{sec:fibo-decoding}, we have briefly seen the strategy for detecting and correcting errors in the codeword. Specifically, the first step is to exploit the checking element $\det M$, 
i.e., if $\det C \not  = (-1)^n \det M$, then $C$ contains one or more errors.

If we are in the second case, we proceed by evaluating the quantities $x_i$'s given in equations \eqref{eq:xi} and we use them as follows:

\begin{enumerate}
	\item if only one $x_i$ is an integer, then $C$ contains only one error and it can be corrected as seen in Section \ref{sec:fibo-decoding};
	\item if all the $x_i$ are not integers, then $C$ contains two or more errors.
\end{enumerate} 

Stakhov \cite{Stakhov} suggested some strategies to manage the second case.
but they should be analyzed and discussed more deeply. 
First of all, we can observe that it is not efficient to check all the hypotheses of double errors and possibly also all the hypotheses of triple errors. It is surely more convenient to exploit some properties that can identify the situation. For this purpose, it is convenient to check if the quantities $\frac{c_1}{c_2}$ and $\frac{c_3}{c_4}$ belong to the approximation interval $\left[ \frac{F_{n+1}}{F_n}, \frac{F_n}{F_{n-1}} \right]$:

\begin{enumerate}
	\item if $\cfrac{c_1}{c_2} \in \left[ \frac{F_{n+1}}{F_n}, \frac{F_n}{F_{n-1}} \right]$ and $\cfrac{c_3}{c_4} \not \in \left[ \frac{F_{n+1}}{F_n}, \frac{F_n}{F_{n-1}} \right]$, then $c_1, c_2$ are correct and $c_3, c_4$ are wrong;
	\item if $\cfrac{c_1}{c_2} \not\in \left[ \frac{F_{n+1}}{F_n}, \frac{F_n}{F_{n-1}} \right]$ and $\frac{c_3}{c_4} \in \left[ \frac{F_{n+1}}{F_n}, \frac{F_n}{F_{n-1}} \right]$, then $c_1, c_2$ are wrong and $c_3, c_4$ are correct;
	\item otherwise, we have one of the following situation
	\begin{enumerate}
		\item $c_1, c_3$ are wrong or $c_2, c_4$ are wrong;
		\item $c_1, c_4$ are wrong or $c_2, c_3$ are wrong;
		\item there are three errors;
		\item there are four errors .
	\end{enumerate}
\end{enumerate}

Four errors can not be corrected. 
Moreover, in this paper, we do not discuss the situation with three errors, 
but only the case where two errors occurred.

\subsection{Correction of two errors not in the same row}\label{sec:twoerr_not_same_row}

In the following we will use this notation:
\begin{itemize}
	\item $M = \begin{pmatrix} m_1 & m_2 \cr m_3 & m_4 \end{pmatrix}$ is the message
	\item $C = \begin{pmatrix} c_1 & c_2 \cr c_3 & c_4  \end{pmatrix}$ is the codeword correspondig to $M$, i.e., $C = M \cdot Q^n$
	\item $\bar C = \begin{pmatrix} \bar c_1 & \bar c_2 \cr \bar c_3 & \bar c_4  \end{pmatrix}$ is the received message that could contain some errors, i.e., $\bar C = C + E$, where $E = \begin{pmatrix} e_1 & e_2 \cr e_3 & e_4  \end{pmatrix}$ is the error matrix, whose entries are integer numbers (if all the entries of $E$ are zero, then the received message is correct, i.e., it coincides with the codeword $C$).
	\item $a = \frac{F_{n+1}}{F_n}$ and $b = \frac{F_n}{F_{n-1}}$
\end{itemize}

Let us suppose that we have received $\bar C$ such that $\cfrac{\bar c_1}{\bar c_2} \not\in \left[ a, b \right]$, $\cfrac{\bar c_3}{\bar c_4} \not\in \left[a, b \right]$ and we are in the presence of two errors. Thus, the possible errors are in the entries $\bar c_1, \bar c_3$ or $\bar c_2, \bar c_4$ or $\bar c_1, \bar c_4$ or $\bar c_2, \bar c_3$, i.e., the errors occur not in the same row of the matrix $\bar C$. In this case, Stakhov \cite{Stakhov} suggested to solve the following Diophantine equations
\begin{equation} \label{eq:lin-diof} x\bar c_4 - \bar c_2y = (-1)^n \det M, \quad \bar c_1x - y\bar c_3 = (-1)^n \det M \end{equation}
\begin{equation} \label{eq:diof-fatt} xy - \bar c_2\bar c_3 = (-1)^n \det M, \quad \bar c_1\bar c_4 - xy = (-1)^n \det M \end{equation}
and choose the solutions that allows to satisfy the condition of approximations for correcting the errors. This approach has several problems. First of all, without any condition on the message $M$, it could be hard to find among the infinite solutions of the previous Diophantine equations those that belong to the approximation interval $[a, b]$. Moreover, the Diophantine equations \eqref{eq:diof-fatt} are hard to solve if $\det M$ is very large, indeed they are equivalent to the factorization problem of an integer number.
If we consider messages $M$ such that $m_1, m_2, m_3, m_4 < F_{n-1}$, thanks to Proposition~\ref{prop:one-error}, we are able to correct the errors without solving any Diophantine equation. 
For instance, let us consider a received message $\bar C$ containing errors in the first column, i.e., $e_1, e_3 \not= 0$ and $e_2=e_4=0$. By Proposition \ref{prop:one-error}, we know that $c_1$ is the only integer such that $\cfrac{c_1}{c_2} \in \left[ a, b \right]$ and we can easily find the integer $e_1$. Indeed, if $\cfrac{\bar c_1}{\bar c_2} = \cfrac{\bar c_1}{c_2}  > b$, then we know that $e_1$ is the smallest integer greater than the rational number $\bar c_1 - \bar c_2 \cdot b = \bar c_1 - c_2 \cdot b$, i.e., 
$$e_1 = \lceil \bar c_1 - c_2 \cdot b \rceil,$$ 
where $\lceil \cdot \rceil$ is the ceiling function. Similarly, we have that 
$$e_3 = \lceil \bar c_3 - c_4 \cdot b \rceil.$$
We can observe that, in all the situations discussed here, where two errors occur not in the same row, the corrections are independent. For instance, in the previous case, for evaluating $e_1$ we have only used $\bar c_1$ and $\bar c_2 = c_2$ and for evaluating $e_3$ we have only used $\bar c_3$ and $\bar c_4 = c_4$. All the cases of two errors that do not occur in the same row can be detected and corrected with similar considerations. We summarize how to evaluate $e_i$ in all the situations that can occur in this scenario:
\begin{itemize}
	\item if $\cfrac{\bar c_i}{c_j} > b$, then $e_i =\lceil \bar c_i - c_j b \rceil$, for $i=1$ and $j=2$ or $i=3$ and $j=4$;
	\item if $\cfrac{c_i}{\bar c_j} > b$, then $e_j =\lfloor \bar c_j - \frac{c_i}{b} \rfloor$, for $i=1$ and $j=2$ or $i=3$ and $j=4$;
	\item if $\cfrac{\bar c_i}{c_j} < a$, then $e_i =\lfloor \bar c_i - c_j a \rfloor$, for $i=1$ and $j=2$ or $i=3$ and $j=4$;
	\item if $\cfrac{c_i}{\bar c_j} < a$, then $e_j =\lceil \bar c_j - \frac{c_i}{a} \rceil$, for $i=1$ and $j=2$ or $i=3$ and $j=4$;
\end{itemize}
We can observe that the previous relations do not use the checking element $\det M$ and they are surely more efficient to evaluate than solving Diophantine equations. We would like to highlight that if the positions of the errors is known in the matrix $C$, then it is not necessary to send the checking element $\det M$ together with the codeword $C$. However, if these positions are not known, then it is necessary to use it to proceed with the correction. 
For instance, if $\cfrac{\bar c_1}{\bar c_2} > b$ and we do not know whether $e_1 = 0$ or $e_2 = 0$, we have to evaluate both $e_1 = \lceil \bar c_1 - \bar c_2 b \rceil$ and $e_2 = \lfloor \bar c_2 - \frac{\bar c_1}{b} \rfloor$. If we use $e_1$ for retrieving $c_1 = \bar c_1 - e_1$, then $\frac{c_1}{\bar c_2}$ belongs to $[a, b]$, but also using $e_2$ for retrieving $c_2 = \bar c_2 - e_2$, we have $\frac{\bar c_1}{c_2}$ and we are not able to retrieve the codeword corresponding to the sent message. In this situation, the only solution is to use the checking element $\det M$ and comparing the determinants.

\subsection{Correction of two errors in the same row}\label{sec:twoerr_same_row}

The case where two errors occur in the same row of the matrix $\bar C$ is the more difficult to manage. In the following, we suppose that the errors are in the entries $\bar c_1$ and $\bar c_2$, i.e., $e_1, e_2 \not=0$ and $e_1 = e_2 = 0$. In this situation, Stakhov \cite{Stakhov} suggests to solve the Diophantine equation $x \bar c_4 - y \bar c_3 = (-1)^n \det M$ and to correct the errors by means of the solutions that belong to the approximation interval. However, this is not sufficient to detect and correct such errors. 

Let $k\in\NN$, we consider two messages $M$ and $M'$ such that
\begin{equation}\label{eq.MeM1}
M = \left(\begin{array}{cc} m_1 & m_2 \\ m_3 & m_4  \end{array}\right) \mbox{ and }
M' = \begin{pmatrix} m_1+k\,m_3 & m_2+k\,m_4 \cr m_3 & m_4  \end{pmatrix}
\end{equation}
whose elements are positive non--zero integers. The encoding matrices are respectively
\begin{equation}\label{eq.CeC1}
C = \begin{pmatrix} c_1 & c_2 \cr c_3 & c_4  \end{pmatrix} \mbox{ and }
C' = \begin{pmatrix} c_1 + k\,c_3 & c_2+k\,c_4 \cr c_3 & c_4  \end{pmatrix} \end{equation}
We denote, as usual, by $\bar C$ and $\bar C'$ the received matrices. Since $\det(M)=\det(M')$, for correcting errors that are in the first row, we should solve the Diophantine equation
$$x \bar c_4 - y \bar c_3 = (-1)^n \det M$$
for both the matrices $\bar C$ and $\bar C'$, which clearly have two different solutions $(c_1, c_2)$ and $(c_1 + k\,c_3, c_2+k\,c_4)$ that satisfy the checking relation and the approximation condition. This means that when we receive the wrong matrices we are not able to correct the errors since we find more than one possible couple that corrects the errors.

To avoid this problem, we have to restrict the space of messages.
\begin{definition}
	Let $M$ be a $2 \times 2$ matrix
	$$
	M = \left(\begin{array}{cc} m_1 & m_2 \\ m_3 & m_4  \end{array}\right) 
	$$ 
	such that $m_1\geq m_3$ and $m_2\leq m_4$ or viceversa, that is, $m_1\leq m_3$ and $m_2\geq m_4$. Then $M$ is \textbf{minimal}.
\end{definition}

Note that if $M$ is minimal, then 
there exists no 
natural number $k$ such that $m_1=a+k\,m_3$ and $m_2=b+k\,m_4$ where $a,b\in\NN$ (or $m_3=a+k\,m_1$ and $m_4=b+k\,m_2$ where $a,b\in\NN$).

Therefore, the space of messages has to contain just minimal matrices. In this way, we are able to correct two errors in the same row since we take the smallest Diophantine equation solutions that verify the checking relations.

Indeed, let us consider the two messages $M,M'$ as in \eqref{eq.MeM1}, such that $M$ is minimal, and let $C$ and $C'$ be the corresponding codewords as in \eqref{eq.CeC1}.
Now suppose the matrices $\bar C$ and $\bar C'$ containing errors in the first row are received. 
We solve the Diophantine equation 
\begin{equation}\label{eq:diof}
xc_4 - yc_3 = (-1)^n \det M
\end{equation}
and we find  two solution $(x_1,y_1)$ and $(x_2,y_2)$ such that
$$
x_1 = F_{n+1} m_1 + F_n m_2 =-\det M x_0 +t_1 c_3\quad
y_1 = F_n m_1 + F_{n-1} m_2  =-\det M y_0 +t_1 c_4 
$$
and
$$
x_2 = -\det M x_0 +t_2 c_3\quad
y_2 = -\det M y_0 +t_2 c_4 
$$
where $t_1,t_2\in\mathbb Z$ and $x_0,y_0$ verify the equation $c_4\,x_0-c_3\,y_0=1$. 
Since
$$
\begin{cases} 
x_2 = F_{n+1} m_1 + F_n m_2 + k\,(F_{n+1} m_3 + F_n m_4) = x_1+k\,c_3=-\det M x_0 +(k+t_1) c_3\cr 
y_2 = F_n m_1 + F_{n-1} m_2+ k\,(F_{n} m_3 + F_{n-1} m_4) = y_1+k\,c_4= -\det M x_0 +(k+t_1)c_4   
\end{cases} 
$$

and $t_2>t_1$, to find the solution $(x_1,y_1)$, i.e., to correct the error and encode the codeword $C$, it is sufficient to take the smallest solution of the previous Diophantine equation \eqref{eq:diof} that verifies the checking relations.


\subsection{Correction of three errors}

Let us consider the case with three errors. 
Suppose that, for example, we received 
\begin{equation}\label{eq:matrix3er}
\overline C = \begin{pmatrix} c_1 &  \bar c_2 \cr \bar c_3 & \bar c_4  \end{pmatrix}
\end{equation}
such that the errors are, for instance, in the entries $\bar c_2,\bar c_3$ and $\bar c_4$, i.e., $e_2, e_3, e_4 \not=0$ and $e_1 = 0$. Therefore $\cfrac{c_1}{\bar c_2}, \cfrac{\bar c_3}{\bar c_4} \not\in \left[a, b \right]$. 
In this situation, Stakhov \cite{Stakhov} suggests to solve the nonlinear Diophantine equation $c_1 z - xy  = (-1)^n \det M$ and to correct the errors by means of the solutions that belong to the approximation interval. However, this nonlinear Diophantine equation can be very hard to solve. Indeed, it is related to equations \eqref{eq:diof-fatt}, when $z$ is fixed, and solving these kind of equations is equivalent to find the factorization of an integer number, which is infeasible for large integers.

We propose a trial-and-error approach to locate the exact position of the errors, and then we correct them.
Precisely, when we know where the errors are, we proceed as follows:
\begin{enumerate}
	\item Consider the row that has only one error. Correct this row using the method described in Section~\ref{sec:twoerr_not_same_row}.
	\item Consider the second row and correct it using the Diophantine equation as shown in Section~\ref{sec:twoerr_same_row}.
\end{enumerate}

For example,
suppose $\overline C = \begin{pmatrix} c_1 &  \bar c_2 \cr \bar c_3 & \bar c_4  \end{pmatrix}$ 
with $\cfrac{ c_1}{\bar c_2}  > b$ is received. 
As in Section~\ref{sec:twoerr_not_same_row}, 
we compute 
$e_2 = \lceil \bar c_2 - c_1 / b \rceil$ and  
recover $c_2= \bar c_2 - e_2$. 
Now, in $\overline C$ we replace $\bar c_2$ with $c_2$, 
obtaining
$
\overline C' = 
\begin{pmatrix}  c_1 & c_2 \cr \bar c_3 & \bar c_4 \end{pmatrix}
$. 
In order to correct $\overline C'$, 
we should solve the Diophantine equation 
$c_1 y - c_2 x = (-1)^n \det M$, 
and choose the smallest positive solution $(x_0,y_0)$ that verifies the checking relation and the approximation condition. 

Note that usually we do not know where the errors are, which means we do not know which is the row with only one error. 
To locate the position with no error we sequentially select each entry of the matrix 
$\overline C$
and try to decode it as described above, 
and we add a check after the first decoding step.
Let us illustrate the possible scenarios following the same example introduced at the beginning of the section.

\paragraph{Case 1 - Wrong row.}

This case happens when we perform step 1 on a row that has two errors.
Suppose we received $\overline C$ as in  \eqref{eq:matrix3er}, 
and that our guess is that the only correct entry in $\overline C$ is $\bar c_4$.
Thus, we try to correct the second row by finding the  error $e_3$. 
We then compute  $\bar c'_3$ and $\bar c_4$ and check if $\cfrac{\bar c'_3}{\bar c_4}\in [a,b]$.
\begin{itemize}
	\item If $\cfrac{\bar c'_3}{\bar c_4}\notin [a,b]$ we understand that we make a mistake and so we restart the procedure changing our guess. For example, we guess that the only correct entry in $\overline C$ is $\bar c_3$,
	and we start again the process to find and correct the errors.
	\item If $\cfrac{\bar c'_3}{\bar c_4}\in [a,b]$ we do not understand that we make a mistake and so we suppose that the solution $(\bar c_3',\bar c_4)$ is correct and we go on with the correction of the other two elements $c_1,\bar c_2$ by solving the corresponding  Diophantine equation, which allow us to find $\bar c_1,\bar c'_2$. 
	Now, to realize if we made a mistake we have to check if $\cfrac{\bar c_1}{\bar c'_2} \in [a,b]$. Note that if $\bar c'_3/\bar c_4\in [a,b]$ then $x_1/y_1\not\in [a,b]$ (see Lemma~\ref{thm:almeno_uno}).
\end{itemize}
Now, let us see in more detail how the decoding works.
As before, suppose we received the matrix \eqref{eq:matrix3er}, but our guess is that the only correct entry is $\bar c_4$.  
Let us consider the case with $\cfrac{\bar c_3}{\bar c_4}  > b$, so that we compute
$$e'_3 = \lceil \bar c_3 - \bar c_4 \cdot b \rceil = e_3 + \Big\lceil -\cfrac{m_3}{ F_{n-1}} -   e_4 \cdot b \Big\rceil = e_3 + h,$$
where $h\in \ZZ$, and 
$$\bar c'_3 = \bar c_3 - e'_3 = c_3 - h.$$
To verify if $\cfrac{\bar c'_3}{\bar c_4}\in [a,b]$ we use Proposition \ref{prop:two-errors} and check, knowing $h$, if $e_4$ is such that
\begin{equation}\label{eq:e4}
\cfrac{-h F_{n-1} - m_3}{F_n} < e_4 < \cfrac{-h F_n + m_4}{F_{n+1}}\,.
\end{equation}

\begin{lemma}\label{thm:e_int}
	Suppose to have received the matrix \eqref{eq:matrix3er} but we guess to have
	\begin{equation}\label{eq:matrix3ercase1}
	\overline C' = 
	\begin{pmatrix} 
	\bar c_1 & \bar c_2 \cr \bar c_3 &  c_4  
	\end{pmatrix}
	\,.
	\end{equation}
	If we compute $\bar c'_3$ and $\bar c_4$ as before we have that $\bar c'_3/\bar c_4\in [a,b]$ if and only if
	\small{$$
	m_3+e_4 F_{n}= \left\{\begin{array}{lll} 
	q F_{n-1}+r, \mbox{ where } r\leq F_{n-1}, && e_4>0 \\ 
	q F_{n-1}+r, \mbox{ where } r\leq F_{n-1}  \mbox{ and  } F_{n}^2+1<F_{n+1}(m_3+r)+F_n m_4, &&  e_4<0\\
	\end{array}\right.
	$$}
\end{lemma}
\begin{proof}
	Let $m\in\ZZ$ and $n\in\NN$, then the ceiling function $\lceil m/n\rceil$ is
	$$ 
	\begin{array}{l} 
	\mbox {if } m<0 \quad \Big\lceil \cfrac{m}{n}\Big\rceil = \left\{\begin{array}{lll} 
	0 && m< n \\ 
	q && m=q n \\ 
	q &&  m=q n +r, \quad r< n \\ 
	\end{array}\right.\\
	\\
	 \mbox {if } m>0\quad \Big\lceil \cfrac{m}{n}\Big\rceil = \left\{\begin{array}{lll} 
	1 && m< n \\ 
	q && m=q n \\ 
	q+1 &&  m=q n +r, \quad r< n \\ 
	\end{array}\right.
	\end{array}
	$$
	Let $\bar c'_3 =  c_3 - h$ and $\bar c_4 = c_4 + e_4$ where 
	$h=\Big\lceil -\cfrac{m_3 + e_4F_n}{ F_{n-1}}\Big\rceil$.
	We consider two different case:
	\begin{itemize}
		\item $e_4>0$. In this case the numerator of the ceiling function is always negative. So we consider three different case:
		\begin{enumerate}
			\item[$(a)$] if $m_3 + e_4F_n < F_{n-1}$ that is impossible.
			\item[$(b)$] if $m_3 + e_4F_n =q F_{n-1}$ than $h=-q$. By Proposition \ref{prop:two-errors} we have that $\bar c'_3/\bar c_4$ belongs in $[a,b]$ if and only if the inequality \eqref{eq:e4} is verify, that is
			$(q F_{n-1} - m_3)/F_n < e_4$
			but in our case $e_4$ is exactly ${q F_{n-1} - m_3}{F_n}$ so $\bar c'_3/\bar c_4\notin [a,b]$.
			\item[$(c)$] if $m_3 + e_4F_n =q F_{n-1}+r$ with $r$ an integer such that $r< F_{n-1}$ than $h=-q$. As before, we consider the inequality \eqref{eq:e4} and the equation $e_4 =(q F_{n-1}+r-m_3)/F_{n}$. Since $m_3,m_4< F_{n-1}$, for one side we have		
			$$
			\qquad\qquad\cfrac{q F_{n-1} - m_3}{F_n} < e_4=\cfrac{q F_n + r - m_4}{F_{n}} \iff m_4<q(F_{n}+F_{n-1})+2F_{n-1}
			$$ 
			that is always true. To the other side we have
			$$
			\quad e_4=\cfrac{q F_n + r - m_4}{F_{n}}<\cfrac{q F_{n} + m_4}{F_{n+1}} \iff r F_{n+1}< q+F_{n+1}F_{n+2}
			$$ 
			that is always true since $r< F_{n-1}$. Therefor, by Proposition \ref{prop:two-errors} we have that $\bar c'_3/\bar c_4\in [a,b]$.
		\end{enumerate}
		\item $e_4<0$. In this case the numerator of the ceiling function is always positive since $m_3< F_{n-1}$. So we have 
		\begin{enumerate}
			\item[$(a)$] if $|e_4|F_n-m_3  < F_{n-1}$ than $e_4=-1$ and $h=1$.  It is simple to check that since $c_3/ c_4\in [a,b]$, then $\bar c'_3/\bar c_4 = (c_3-1)/(c_4-1) \notin [a,b]$. 
			\item[$(b)$] if $|e_4|F_n-m_3  = q F_{n-1}$ than $h=q$. Note that $\bar c'_3/\bar c_4<b$ if and only if 
			$$
			F_{n-1}(F_{n+1}m_3+F_n m_4-q) < F_n(F_{n}m_3+F_{n-1} m_4-|e_4|)$$ 
			that is, $(qF_{n-1}+m_3)(1-F_n)>0 $ which is impossible. 
			So $\bar c'_3/\bar c_4\notin [a,b]$. 
			\item[$(c)$] if $|e_4|F_n-m_3  =q F_{n-1}+r$ with $r$ an integer s.t. $r< F_{n-1}$ than $h=q+1$.  As before, we consider the inequality \eqref{eq:e4}. So for one side we have
			$$
			\qquad \cfrac{-(q+1) F_{n-1} - m_3}{F_n} < -|e_4|=\cfrac{-q F_{n-1}-r-m_3}{F_{n}} \iff r< F_{n-1}
			$$ 
			that is always true. 
			To the other side, considering $|e_4|F_n =q F_{n-1}+r+m_3$, we have
			\begin{align*}
				\qquad -|e_4|<\cfrac{-(q+1) F_{n} + m_4}{F_{n+1}} 
				& 
				\iff 
				F_{n}(F_{n+1}|e_4|+m_4)> (q+1) F^2_{n} 
				\\
				& 
				\iff 
				F_{n+1}(m_3+r)+F_n m_4 >F_{n}^2+1
			\end{align*}
			and we obtain the claim.
		\end{enumerate}
	\end{itemize}
	\end{proof}
Let us suppose that we are in the situation of Lemma \ref{thm:e_int}, that is,  $e_4$ belongs in that range. So we guess that we correct properly and we go on for the next step computing the other two elements $\bar c_1,c_2$ using the Diophantine equation. We have to find $(x,y)$ such that
\begin{equation}\label{eq:x1y1}
x \bar c_4- y\bar c_3' = -\det M.
\end{equation}
The solutions of this equation are
\begin{equation}
\begin{cases} 
x_1 = -\det M x_0 +t \bar c'_3\cr 
y_1 = -\det M y_0 +t \bar c_4
\end{cases} 
\end{equation}
where $t\in\mathbb Z$ and $x_0,y_0$ verify the equation $\bar c_4\,x_0-\bar c'_3\,y_0=1$.

\begin{lemma}\label{thm:almeno_uno} Suppose to have received the matrix \eqref{eq:matrix3er} but we guess to have \eqref{eq:matrix3ercase1}. We compute $\bar c'_3$ and $\bar c_4$ as before and suppose that $\bar c'_3/\bar c_4\in [a,b]$. Then we  find $(x_1,y_1)$ as solution of \eqref{eq:x1y1}. So 
	$$\mbox{if } \cfrac{\bar c'_3}{\bar c_4}\in [a,b] \mbox{ then } \cfrac{x_1}{y_1}\not\in [a,b].$$
\end{lemma}
\begin{proof}
	Since $\bar c'_3/\bar c_4\in [a,b]$ by Lemma \ref{thm:e_int} we have two cases:
	\begin{itemize}
		\item $e_4>0$ such that $e_4F_n =q F_{n-1}+r-m_3$ with $r< F_{n-1}$ and $h=-q$. Because $(x_1,y_1)$ is a  solution of \eqref{eq:x1y1} we have 
		$$
		\cfrac{x_1}{y_1}=\cfrac{y_1(c_3+q)-\det M}{y_1(c_4+e_4)}
		$$
		We want to verify that $x_1/y_1\not\in [a,b]$:
		\begin{enumerate}
			\item[(a)] $x_1/y_1> a$ if and only if $\det M<0$. In fact, 
			$$x_1/y_1>F_{n+1}/F_n \iff F_n y_1(c_3 +q) - F_n  \det M > F_{n+1} y_1(c_4+e_4).$$
			Because $e_4$ verify \eqref{eq:e4},
			$F_{n+1} y_1(c_4+e_4)<F_{n+1}y_1c_4+F_n y_1 q + y_1 m_4$ but to the other side when $\det M>0$ we have that
			$$ F_n y_1(c_3 +q) - F_n  \det M > F_{n+1}y_1c_4+F_n y_1 q + y_1 m_4,$$
			that is
			$$
			\qquad y_1 F_n(F_{n+1}m_3 + F_n m_4) - F_n  \det M > y_1 F_{n+1}(F_{n}m_3 + F_{n-1} m_4) +m_4y_1,
			$$
			thus,  $F_n\det M >0$.  This mean that $x_1/y_1\notin[a,b]$ when $\det M>0$.
			\item[(b)] $x_1/y_1<b$ if and only if $\det M>0$. In fact, 
			$$\qquad x_1/y_1<F_{n}/F_{n+1} \iff F_{n-1} y_1(c_3 +q) - F_{n-1}\det M <F_{n} y_1(c_4+e_4).$$
			Because $e_4$ verify \eqref{eq:e4},
			$F_{n} y_1(c_4+e_4)>F_{n}y_1c_4+F_{n-1} y_1 q - y_1 m_3$ but when $\det M>0$ we have 
			$$F_{n}y_1c_4+F_{n-1} y_1 q - y_1 m_3> F_{n-1} y_1(c_3 +q) - F_{n-1}\det M$$
			in fact
			$$\qquad \qquad y_1 F_n(F_{n}m_3 + F_{n-1} m_4)> y_1 F_{n-1}(F_{n+1}m_3 + F_n m_4)  - F_{n-1}\det M,$$ 
			thus, $F_{n-1}\det M >0$. 
			This mean that $x_1/y_1\notin[a,b]$ when $\det M<0$.
		\end{enumerate}
		In similar way we can prove the following case.
		\item $e_4<0$ such that $|e_4|F_n=q F_{n-1}+r+m_3$ with $r< F_{n-1}$ and $F_{n}^2+1<F_{n+1}(m_3+r)+F_n m_4$  than $h=q+1$. Because $(x_1,y_1)$ is a  solution of \eqref{eq:x1y1} we have 
		$$
		\cfrac{x_1}{y_1}=\cfrac{y_1(c_3-(q+1))-\det M}{y_1(c_4-|e_4|)}
		$$
		We want to verify that $x_1/y_1\not\in [a,b]$:
		\begin{enumerate}
			\item[(a)] $x_1/y_1> a$ if and only if $\det M<0$. In fact, 
			$$\qquad \qquad x_1/y_1>F_{n+1}/F_n \iff F_n y_1(c_3 -(q+1)) - F_n  \det M > F_{n+1} y_1(c_4-|e_4|).$$
			Because $e_4$ verify \eqref{eq:e4},
			$F_{n+1} y_1(|e_4|-c_4)>(q+1)F_{n}y_1+ y_1 m_4-F_{n+1} y_1 c_4$ but to the other side when $\det M<0$ we have that
			$$ (q+1)F_{n}y_1+ y_1 m_4-F_{n+1} y_1 c_4> F_n y_1(c_3 -(q+1)) - F_n  \det M $$
			in fact
			$$ 
			\qquad \qquad y_1 m_4-F_{n+1} y_1(F_{n}m_3 + F_{n-1} m_4) > y_1 F_{n+1}(F_{n+1}m_3 + F_n m_4)  - F_n  \det M,
			$$ 
			thus, $F_n\det M < 0$.
			This mean that $x_1/y_1\notin[a,b]$ when $\det M>0$.
			\item[(b)] $x_1/y_1<b$ if and only if $\det M>0$. In fact, 
			$$\qquad \qquad \qquad x_1/y_1<F_{n}/F_{n-1} \iff F_{n-1} y_1(c_3 -(q+1)) - F_{n-1}  \det M < F_{n} y_1(c_4-|e_4|).$$
			Because $e_4$ verify \eqref{eq:e4},
			$F_{n} y_1(|e_4|-c_4)<(q+1)F_{n-1}y_1+ y_1 m_3-F_{n} y_1 c_4$ but when $\det M>0$ we have 
			$$\qquad (q+1)F_{n-1}y_1+ y_1 m_3-F_{n} y_1 c_4<F_{n-1} y_1((q+1)-c_3) + F_{n-1}  \det M$$
			in fact
			$$
			\qquad \qquad y_1 m_3- F_n y_1(F_{n}m_3 + F_{n-1} m_4)< -F_{n-1}y_1 (F_{n+1}m_3 + F_n m_4)  + F_{n-1}  \det M,
			$$ 
			$F_{n-1}\det M >0$. This mean that $x_1/y_1\notin[a,b]$ when $\det M<0$.
		\end{enumerate}
	\end{itemize}
	\end{proof}

\paragraph{Case 2 - Correct row, wrong element.}

This case is similar to the previous one.
Suppose to receive $\overline C$ as \eqref{eq:matrix3er}
but we guess to have received 
\begin{equation}
\overline C' = \begin{pmatrix}  \bar c_1 & c_2 \cr \bar c_3 &  \bar c_4  \end{pmatrix}
\end{equation}
We start to correct the first row trying to find the  error $e_1$.\\
Suppose to have $\cfrac{c_1}{\bar c_2}  > b$ so we compute
$$e'_1 = \lceil c_1 - \bar c_2 \cdot b \rceil = \lceil - \cfrac{m_1}{F_{n-1}} - e_2 \cdot b \rceil = h,$$
where $h\in \ZZ$, and 
$$\bar c_1 = c_1 - e'_1 = c_1 - h.$$
We are exactly in the previous case so we can proceed as before.


\section{Computing the redundancy of Fibonacci codes} \label{sec:redundancy}
In this section, we give an estimation of the redundancy inserted by the coding method. In \cite{StakhovBook}, the author only discussed the redundancy due to the transmission of $\det M$ together with the encoded message, but nothing is said about the redundancy inserted by the matrix multiplication \eqref{eq:coding}.
\begin{theorem}
	Let us consider the message space ${\mathcal M} = \{0,1\}^k$, with $k=4h$ for some integer $h$, 
	so that we can split a message in 4 block of equal bit size.
	For a fixed $n$, then a codeword of a Fibonacci code with generator matrix $Q^n$ lies in the space ${\mathcal C} = \{0,1\}^l$, with 
	$$l = 
	\lfloor 
	(4n+2)\log_2 \varphi - 2\log_2 5 + 3k/2 + 5
	\rfloor \,.$$
\end{theorem}
\begin{proof}
	Let $\varphi = \frac{\sqrt{5}+1}{2}$ be the Golden Ratio.
	By Binet's formula \cite[Section 1.2.8]{knuth1968art}
	\footnote{In the reference this formula is attributed to De Moivre.},
	we have that $F_n = \frac{\varphi^n - (-\varphi)^{-n}}{\sqrt{5}} \approx \round{\frac{\varphi^n}{\sqrt{5}}}$, 
	where the operator $\round{.}$ approximates to the nearest integer.
	We can then deduce that
	$
	\log_2 F_n \approx n \log_2 \varphi - \frac{\log_2 5}{2}.
	$
	
	Let 
	$
	M \in {\mathcal M}$, then $C = M \times Q^n = 
	\begin{pmatrix}
	c_1 & 
	c_2 \cr 
	c_3 & 
	c_4  
	\end{pmatrix}
	=
	\begin{pmatrix}
	F_{n+1} m_1 + F_n m_2 & 
	F_n m_1 + F_{n-1} m_2 \cr 
	F_{n+1} m_3 + F_n m_4 & 
	F_n m_3 + F_{n-1} m_4  
	\end{pmatrix}\,.
	$
	Then the number of bits of $C$ is  
	$
	\lfloor 
	(4n+2)\log_2 \varphi - 2\log_2 5 + k + 4
	\rfloor
	$.
	
	To conclude, we need to consider that the determinant of $M$ has to be sent in order to decode, and since the determinant is maximal when $m_1 = m_4 = 2^h - 1$ and $m_2 = m_3 = 0$, then no more than $2h+1$ bits are needed to represent the determinant of $M$, where one bit is needed for the sign.
	\end{proof}
In practice, the value of the Golden Ratio $\varphi = \frac{\sqrt{5}+1}{2}$ should be approximated in such a way that the number of binary digit to represent it is greater than the number of binary digits of $n$.
In practice, the number of bits needed to represent the matrix $C$ is 
$2.8 n + k $,
since $$\log_2 \varphi \approx 0.694241913630617301738790266899$$
and  $$\frac{\log_2 5}{2} \approx 1.16096404744368117393515971474.$$
This means that $l$ can be approximated by $2.8n + 1.5k$. 
We can conclude with the following
\begin{corollary}
	The Fibonacci code with generator matrix $Q^n$ and encoding messages of length $k=4h$, for an integer $h$, has redundancy $2.8n + 0.5k$.
\end{corollary}

\section{Conclusions and future works} \label{sec:conc}

We have provided a more detailed discussion on how to detect and correct double errors in the special case of $2 \times 2$ generator matrices in the error correcting codes introduced by Stakhov, i.e. 
1-Fibonacci error correcting codes.
A similar analysis should be provided for $p \times p$ generator matrices, and for a larger amount of errors.
We also provided an explicit formula to compute the redundancy of 1-Fibonacci error correcting codes.
In the future, also a metric for this codes should be introduced, 
in order to define the concept of codewords weight and distance.
Due to the particular kind of errors that the 1-Fibonacci error correcting codes can correct,  it is hard to see in which real scenario they could be applied as an error correction tool. 
On the other hand, the compact representation of their generator matrix may be useful to build one way functions for cryptography,
where often standard codes offer very large keys, encryptions or signatures. 
For this reason, more research on the actual complexity of the decoding algorithm should be conducted.

\bibliographystyle{amsalpha}
\bibliography{biblio-fibo}

\end{document}